\documentclass[10pt]{amsart}

\usepackage{graphicx}
\usepackage{comment}
\usepackage{amsfonts,amsmath,latexsym,amssymb,amsthm}
\usepackage{setspace}

\newcounter{theoremcounter}
\newcounter{lemmacounter}

\newcounter{dummycounter}
\newcounter{propcounter}
\newcounter{corcounter}

\newcounter{emptycounter}
\newcounter{defcounter}

\newcounter{examcounter}

\newtheorem{theorem}[theoremcounter]{Theorem}

\newtheorem{lemma}[lemmacounter]{Lemma}

\newtheorem{proposition}[propcounter]{Proposition}
\newtheorem{corollary}[corcounter]{Corollary}

\newtheorem{definition}[defcounter]{Definition}

\newtheorem{example}[examcounter]{Example}

\numberwithin{equation}{section}
\numberwithin{lemmacounter}{section}
\numberwithin{propcounter}{section}
\numberwithin{corcounter}{section}
\numberwithin{conjcounter}{section}
\numberwithin{theoremcounter}{section}
\numberwithin{probcounter}{section}

\newcounter{eqncounter}

\numberwithin{equation}{eqncounter}

\def\IR{\mathbf R}
\def\IC{\mathbf C}
\def\IZ{\mathbf Z}
\def\IN{\mathbf N}

\def\T{T}

\def\valpha{{\mbox{\boldmath $\alpha$}}}

\def\vNull{{\mbox{\boldmath $0$}}}

\def\L{\mathcal{L}}

\def\A{{\mathfrak{A}}}
\def\C{{\mathfrak{C}}}
\def\c{{\mathfrak{c}}}

\def\E{\IR^N}

\def\F{{\tilde{F}}}

\def\m{m}
\let\rho\varrho

\def\Ti{{T_{{\bf i}}}}

\def\vbeta{{\mbox{\boldmath $\beta$}}}

\def\Q{{\mbox{\boldmath $Q$}}}
\def\vm{{\mbox{\boldmath $m$}}}

\def\Vol{\textup{Vol}}

\def\Lam1{\Lambda_1}

\def\C0{C_{\bf{0}}}

\def\Da{D}
\def\SIi0{S_{F}(\i)}

\def\SI0i{S_{F}({\bf{0}})}

\def\p0{\psi\Ti(\Lambda)}

\def\diam{\text{diam}}

\def\c1{c_1}

\def\ka{\kappa}

\def\v{v}

\def\r{r}

\def\c{c}

\def\cone{c_2}
\def\ctwo{c_3}
\def\cthree{c_1}
\def\cfour{c_5}
\def\cfive{c_7}
\def\copnorm{c_{6}}

\def\ccountlatticepts{c_{4}}

\def\con{\gamma}

\def\divfct(|P|+R){(|P|+R)^\varepsilon}

\def\Fd{F(d)}

\def\Fo{F(1)}

\def\App{\varphi}

\def\Rt{\Ab}

\def\i{{{\bf i}}}
\def\v0{{{\bf 0}}}
\def\vv0{\underline{{\bf 0}}}

\def\vx{{\bf x}}
\def\vy{{\bf y}}
\def\vz{{\bf z}}
\def\vp{{\bf p}}
\def\vvx{\underline{\bf x}}
\def\vvy{\underline{\bf y}}
\def\vvz{\underline{\bf z}}

\def\vp{{\bf p}}
\def\vq{{\bf q}}

\def\Nm{\textup{Nm}_{\vbeta}}
\def\Nmm{\textup{Nm}}

\def\t{{t}}
\def\s{{s}}
\def\N{{N}}

\def\me{m_1}
\def\mn{m_n}
\def\mi{m_i}

\def\S{\mathcal{S}}

\def\C{C}

\def\ro{\rho}
\def\la{\lambda}
\def\La{\Lambda}

\def\Qi{Q_i}
\def\Qo{Q_1}
\def\Qn{Q_n}
\def\Qm{Q_{max}}
\def\Qmin{Q_{min}}

\def\R{R}
\def\Rz{\overline{Q}}

\def\Z{Z_{\Q}}
\def\Zj{Z_{\Q_j}}

\def\Cx{B_{\vvx_j}}

\def\vyi{\vy_i}
\def\Byi{B_{\vyi}}

\def\mul{\mu(\La,\Qm)}

\def\gcd{\text{gcd}}
\def\Ab{G}
\def\Ac{B}
\def\Ad{H}
\def\dS{n^{1/2}}

\def\thi{\theta_i}
\def\tho{\theta_1}
\def\thn{\theta_n}
\def\thm{\theta_{min}}

\def\k{k}

\def\a{a_M}
\def\cet{c}

\def\ka{\kappa}
\def\F{\mathcal{F}}
\def\Aut{\text{GL}}

\def\LEZ{\mathcal{E}_{\La}(\Z)}
\def\LEZi{\mathcal{E}_{\La}(\Zj)}
\def\x{\mathfrak{Q}}

\def\Mat{\text{Mat}}

\begin{document}

\title{Weak admissibility, primitivity, o-minimality, and Diophantine approximation}

\author{Martin Widmer}

\address{Department of Mathematics\\ 
Royal Holloway, University of London\\ 
TW20 0EX Egham\\ 
UK}

\email{martin.widmer@rhul.ac.uk}

\date{\today}



\maketitle

\tableofcontents

\begin{abstract}
We generalise M. M. Skriganov's notion of weak admissibility for lattices to include standard lattices 
occurring in Diophantine approximation and algebraic number theory,
and we prove estimates for the number of lattice points in sets such as aligned boxes. Our result improves
on Skriganov's celebrated counting result if the box is sufficiently distorted, the lattice is not admissible, and, e.g., symplectic or orthogonal.
We establish a criterion under which our error term is sharp, and we provide
examples in dimensions $2$ and $3$ using continued fractions.
We also establish a similar counting result for primitive lattice points, and apply the latter
to the classical problem of Diophantine approximation with primitive points as studied
by Chalk, Erd\H{o}s, and others. Finally, we use o-minimality to describe large classes of
sets 
to which our counting results apply.
\end{abstract}

\section{Introduction}\label{intro}
In this article we generalise Skriganov's notion of (weak) admissibility for lattices to include standard lattices 
occurring in Diophantine approximation and algebraic number theory (e.g., ideal lattices),
and we prove a sharp estimate for the number of lattice points in sets such as aligned boxes.
Our result applies when the lattice is weakly admissible, whereas Skriganov's result requires the
dual lattice to be weakly admissible (in his stronger sense). If the lattice is symplectic or orthogonal\footnote{The lattice $\La=A\IZ^N$ is symplectic (or orthogonal) if $A\in \Aut_N(\IR)$ is symplectic (or orthogonal).} and weakly admissible 
then both results apply, and our error term is better, provided the lattice is not admissible and the
box is sufficiently distorted.
Our error term also has a good dependence on the geometry of the lattice which allows us to apply a M\"obius inversion 
to get a similar estimate for primitive lattice points.
The motivation for this comes from a classical Diophantine approximation result \cite{ChalkErdos1959} due to Chalk and Erd\H{os} from 1959 for numbers;
it appears that our result is the first one in higher dimensions.
We also make modest progress on a conjecture of Dani, Laurent, and Nogueira \cite{DaniLaurentNogueira2014,LaurentNogueira2012} on an inhomogeneous Khintchine Groshev type result for primitive points.
Finally, we use o-minimality, a notion from model theory, to describe large classes of sets to which
our counting results apply. The usage of o-minimality to asymptotically count lattice points has been initiated by Barroero and the author \cite{BarroeroWidmer}, and
\cite[Theorem 1.3]{BarroeroWidmer} has already found various applications (see, e.g., \cite{Barroero1,Barroero2,Frei1,Frei2,Frei3,Frei4}). Here we  further develop this idea but we use o-minimality in a different way.\\

Next  we shall state the simplest special case of Theorem \ref{Thm1}, and compare it to Skriganov's result \cite[Theorem 6.1]{Skriganov} (more precisely, Technau and the authors generalisation \cite[Theorem 1]{TechnauWidmer} to inhomogeneously expanding boxes). 
Let $\Gamma\subset \IR^N$. Following Skriganov we define $\nu(\Gamma,\rho):=\inf\{|x_1\cdots x_N|^{1/N}; \vvx\in \Gamma\backslash\{\vv0\}, |\vvx|<\rho\}$, and 
we say a lattice $\La$ in $\IR^N$ is weakly admissible if $\nu(\La,\rho)>0$ for all $\rho>0$ and admissible if $\lim_{\rho\rightarrow \infty}\nu(\La,\rho)>0$.
Let $\Z$ be a translate of the box $[-Q_1,Q_1]\times\cdots \times[-Q_N,Q_N]$, and write $\Qm$ for the maximal $Q_i$, and $\Rz$ for their geometric mean.
We set $\LEZ:=\left|\#(\La\cap \Z)-\Vol\Z/\det\La\right|$.
\begin{theorem}\label{Thm0}
Suppose $\La$ is a weakly admissible lattice in $\IR^N$. Then we have
\begin{alignat*}3
\LEZ
\ll_N \inf_{0<\Ac\leq \Qm}\left(\frac{\Rz}{\nu(\La,\Ac)}+\frac{\Qm}{\Ac}\right)^{N-1}.
\end{alignat*}
\end{theorem}
Suppose $\La$ is unimodular. Skriganov \cite[Theorem 6.1]{Skriganov} proved error estimates for  homogeneously expanding
aligned boxes (and more generally certain polyhedrons), provided the dual lattice $\La^\perp$ (with respect to the standard inner product)
is weakly admissible (see also \cite[(1.11) Theorem 1.1]{Skriganov1994} for a precursor of this result for admissible lattices). As shown in \cite[Theorem 1]{TechnauWidmer} his method also leads to results for inhomogeneously expanding
aligned boxes (provided $\La^\perp$ is weakly admissible) of the form\footnote{In the above setting our definition of $\nu(\cdot,\cdot)$ is the $N$-th root of  Skriganov's and the one in \cite{TechnauWidmer}.}
\begin{alignat}3\label{Skriganovbound}
\LEZ
\ll_N \frac{1}{\nu(\La^\perp, (\Rz/\Qmin)^*)^N}\inf_{\rho> \gamma_N^{1/2}} \left(\frac{\Rz^{N-1}}{\sqrt{\rho}}+\frac{r^{N-1}}{\nu(\La^\perp, 2^r\Rz/\Qmin)^N}\right),
\end{alignat}
where $\gamma_N$ denotes the Hermite constant, $r=N^2+N\log(\rho/\nu(\La^\perp,\rho\Rz/\Qmin))$,  and $(\Rz/\Qmin)^*=\max\{\Rz/\Qmin,\gamma_N\}$.
If $\La$ is admissible (which implies that $\La^\perp$ is admissible) then Skriganov's bound becomes $\ll_\La (\log \Rz)^{N-1}$ which conjecturally is sharp.

Let us know suppose that $\La$ is weakly admissible but not admissible. Technau and the author \cite[Theorem 2]{TechnauWidmer} have shown that in general, even if $\La$ and $\La^\perp$ are both weakly admissible, there is no way to bound 
$\nu(\La,\cdot)$ in terms of $\nu(\La^\perp,\cdot)$. This indicates the complementary aspect of Theorem \ref{Thm0} and  (\ref{Skriganovbound}). However, if 
$\La=A\IZ^N$ with, e.g., a symplectic or orthogonal matrix $A$ then $\nu(\La,\cdot)=\nu(\La^\perp,\cdot)$ by \cite[Proposition 1]{TechnauWidmer},
and we can directly compare our result with Skriganov's; note also that for $N=2$ every unimodular lattice is symplectic (cf.  \cite[Remark after Proposition 1]{TechnauWidmer}). 
Using that $\Rz/\Qmin\geq (\Qm/\Rz)^{1/(N-1)}=:\x$ and that $r\geq -N\log(\nu(\La^\perp,\x)$ we find the following crude lower bound\footnote{We are only interested in ``sufficiently distorted'' boxes, and so we can assume $\x>\gamma_N^{1/2}$.} for  the right hand-side of (\ref{Skriganovbound})
\begin{alignat}3\label{Skriganovbound1}
\left(\nu(\La, \x)\nu(\La, (\nu(\La,\x)^{-N\log 2}\x)\right)^{-N}.
\end{alignat}
Choosing $B=\Qm/\Rz=\x^{N-1}$ we see that the error term in Theorem \ref{Thm0} is bounded from above by
\begin{alignat}3\label{Thm1bound}
\ll_N\Rz^{N-1}\nu(\La, \x^{N-1})^{-(N-1)}.
\end{alignat}
In particular, if $N=2$ then our error term is better whenever $\nu(\La,\Qm/\Rz)^{-3}$ is larger than a certain multiple of $(\Vol\Z)^{1/2}$,
so if the box is sufficiently distorted in terms of $\nu(\La, \cdot)$ and the volume of the box
(note that for $\nu(\La, \Qm/\Rz)^{-1}=o(\Rz)$ as $\Rz$ tends to infinity, we still get asymptotics).
Also for arbitrary $N$ our error term is better when the box is sufficiently distorted in terms of $\nu(\La, \cdot)$ and the volume of the box,
and provided $\nu(\La,\rho)$ decays faster than $\rho^{-1/\log 2}$ or sufficiently slowly, e.g., like a negative power of $\log \rho$. The latter happens for almost every unimodular lattice
(cf. \cite[Lemma 4.5]{Skriganov}), and with $\La=A\IZ^N$ also for almost every\footnote{In the sense of the Haar measure on $SO_N(\IR)$.}  matrix $A\in SO_N(\IR)$ 
(cf. \cite[Lemma 4.3]{Skriganov}), and, as mentioned before, for these $\La$ we also have $\nu(\La,\cdot)=\nu(\La^\perp,\cdot)$.\\

Another significant difference between our and Skriganov's error term concerns the 
dependence on the lattice.  If we replace  $\Z$ by $k^{-1}\Z$ (or equivalently replace $\La$ by $k\La$ and fix $\Z$) then the lower bound (\ref{Skriganovbound1}) of the error term 
in (\ref{Skriganovbound}) remains the same.  
On the other hand  the upper bound (\ref{Thm1bound})
of  the error term in Theorem \ref{Thm0} decreases by a factor $k^{-N+1}$. This improvement allows 
us to sieve for coprimality, and thus to prove asymptotics for the number of primitive lattice points. \\

\section{Generalisation of weak admissibility and statement of the results}
\subsection{Generalised weak admissibility}
Let $\S=(\vm,\vbeta)$, where  $\vm=(\me,\ldots,\mn)\in \IN^n$, $\vbeta=(\beta_1,\ldots,\beta_n)\in (0,\infty)^n$,
and  $n\in \IN=\{1,2,3,\ldots\}$. We write 
$\vx_i$ for the elements in $\IR^{\mi}$ and $\vvx=(\vx_1,\ldots,\vx_n)$
for the elements in $\IR^{\me}\times\cdots\times\IR^{\mn}=\IR^N$,
where 
\begin{alignat*}3
\N&:=\sum_{i=1}^{n}\mi.
\end{alignat*}
We will always assume that $N>1$.
We set
\begin{alignat*}3
\t&:=\sum_{i=1}^{n}\beta_i.
\end{alignat*}
We use $|\cdot|$ to denote the Euclidean norm, and we write 
\begin{alignat*}3
\Nm(\vvx):=\prod_{i=1}^{n}|\vx_i|^{\beta_i}
\end{alignat*}
for the multiplicative $\vbeta$-norm on $\E$ induced by $\S$.  
Let $\C\subset \E$ be a coordinate-tuple subspace, i.e.,
\begin{alignat*}3
\C=\{\vvx\in \E; \vx_i=\vNull\; (\text{for all }i\in I)\},
\end{alignat*}
where $I\subset \{1,\ldots,n\}$. We fix such a pair $(\S,\C)$, and for $\Gamma\subset \E$ and $\ro>0$ we define the quantities
\begin{alignat*}3
\nu(\Gamma,\ro)&:=\inf\{\Nm(\vvx)^{1/\t}; \vvx\in \Gamma\backslash \C, |\vvx|<\ro\},\\
\Nm(\Gamma)&:=\lim_{\rho\rightarrow \infty}\nu(\Gamma,\ro).
\end{alignat*}
As usual we always interpret $\inf\emptyset =\infty$ and $\infty>x$ for all $x\in \IR$. The above quantities in the special case when $\C=\{\vv0\}$
and  $\mi=\beta_i=1$ $(\text{for all }1\leq i\leq n)$ were introduced by Skriganov
in \cite{Skriganov1994, Skriganov}.
By a lattice in $\IR^N$ we always mean a lattice of rank $N$.
\begin{definition}\label{defadm}
Let $\La$ be a lattice in $\IR^N$.
We say $\La$ is {\em weakly admissible} for $(\S,\C)$ if
$\nu(\La,\ro)>0$ for all $\ro>0$.
We say $\La$ is {\em admissible} for $(\S,\C)$ if $\Nm(\La)>0$.
\end{definition}
Note that weak admissibility for a lattice in $\IR^N$ depends only on the choice of $\C$ and $\vm$ whereas admissibility depends on $\C$ and $\S=(\vm,\vbeta)$.
Also notice that a lattice $\La$ in $\IR^N $ is weakly admissible (or admissible) in the sense of Skriganov \cite{Skriganov} if and only if 
$\La$ is  weakly admissible (or admissible) for $(\S,\C)$ with $\C=\{\vv0\}$ and $\mi=\beta_i=1$ $(\text{for all }1\leq i\leq n)$. 
Let us give some examples to illustrate that our notion of weak admissibility captures new interesting cases
not covered by Skriganov's notion of weak admissibility.

\begin{example}
Let $\Theta\in \Mat_{r\times s}(\IR)$ be a matrix with $r$ rows and $s$ columns and consider\footnote{Despite the row  notation we treat the
vectors as column vectors.}
\begin{alignat}3\label{Dioappexample}
\La=\begin{bmatrix} I_r & \Theta \\ \vNull & I_s \end{bmatrix} \IZ^{r+s}=\{(\vp+\Theta\vq,\vq); (\vp,\vq)\in \IZ^r\times\IZ^s\}.
\end{alignat}
We take $n=2$, $\me=r$, $m_2=s$ and $\C=\{(\vx_1,\vx_2); \vx_2=\vNull\}.$ 
Then the lattice $\La$ is weakly admissible for $(\S,\C)$ (for every choice of $\vbeta$)
if $\vp+\Theta\vq\neq \vNull$ for every $\vq\neq \vNull$. If $\vbeta=(1,\beta)$ then $\La$ is admissible for $(\S,\C)$ if 
we have
\begin{alignat}3\label{Dioapp}
|\vp+\Theta\vq|{|\vq|}^{\beta}\geq c_\La
\end{alignat}
for every $(\vp,\vq)$ with $\vq\neq \vNull$ and some fixed $c_\La>0$. 
The above lattice $\La$ naturally arises when considering Diophantine approximations for the matrix $\Theta$ (cf. Corollary \ref{Cor2}).
Recall that the matrix $\Theta$ is called badly approximable
if (\ref{Dioapp}) holds true with $\beta=s/r$. 
W.~M.~Schmidt \cite{Schmidt1969} has
shown that the Hausdorff dimension of the set of badly approximable matrices is full, i.e., $rs$.
\end{example}
Another example comes from the Minkowski-embedding of, e.g., an ideal in a number field.
\begin{example}
Suppose $K$ is a number field  with $r$ real and 
$s$ pairs of complex conjugate embeddings. Let $\sigma:K\rightarrow \IR^r\times\IC^s$ be the Minkowski-embedding, and identify
$\IC$ in the usual way with $\IR^2$. Set n=r+s, $\C=\{\vv0\}$, $\mi=\beta_i=1$ for $1\leq i\leq r$, and   $\mi=\beta_i=2$ for $r+1\leq i\leq r+s$.
Now let $\A\subset K$ be a free $\IZ$-module of rank $N=r+2s$. Then $\La=\sigma\A$ is admissible in $(\S,\C)$.
In particular, this generalises the examples of Skriganov for totally real number fields to arbitrary number fields $K$.
Unlike in Skriganov's setting we can also consider cartesian products of such modules $\A_j$ by using the embedding $\sigma:K^p\rightarrow \IR^{pr}\times\IC^{ps}$
that sends a tuple $\valpha$ to $(\sigma_1(\valpha),\ldots,\sigma_{r+s}(\valpha))$. Now $\mi$ is $p$ if $\sigma_i$ is real and $2p$ otherwise 
while $n$ and $\beta_i$ remain unchanged. Again we get that $\La=\sigma(\A_1\times\cdots\times \A_p)$ is an admissible lattice in $(\S,\C)$.
\end{example}

\subsection{Generalised aligned boxes}
Now we introduce the sets in which we count the lattice points.
Essentially these are the sets that are distorted only in the directions of the coordinate axes. Let $(\S,\C)$ be given, and recall that $\C=\C_I$.

For $\Q=(Q_1,\ldots,Q_n)\in (0,\infty)^n$ we consider the $\vbeta$-weighted geometric mean
\begin{alignat*}3
\Rz=\left(\prod_{i=1}^{n}\Qi^{\beta_i}\right)^{1/\t},
\end{alignat*}
and we assume throughout this note that
\begin{alignat}3\label{Qorder}
\Qi\leq \Rz \;(\text{for all } i\notin I).
\end{alignat}
We set
\begin{alignat*}3
\Qm:=\max_{1\leq i\leq n}\Qi,\\
\Qmin:=\min_{1\leq i\leq n}\Qi.
\end{alignat*}

For $\ka>0$ and $M\in \IN$ we introduce the family of sets
\begin{alignat*}3
\F_{\ka,M}:=\{S\subset \IR^N; \partial(AS)\in \text{ Lip}(N,M,\ka\cdot \diam(AS))\; \forall A\in \Aut_N(\IR)\}.
\end{alignat*}

Here $\Aut_N(\IR)$ denotes the group of invertible $N\times N$-matrices with real entries, $\diam(\cdot)$ denotes the diameter, $\partial(\cdot)$ denotes the topological boundary, 
and the notation Lip-$(\cdot,\cdot,\cdot)$ is explained in Definition \ref{Lip} Section \ref{countingprinciples}.

It is an immediate consequence of \cite[Theorem 2.6]{WidmerLNCLP} that every bounded convex set in $\IR^N$ lies
in $\F_{\ka,M}$ for $\ka=16N^{5/2}$ and $M=1$. We will also show (Proposition \ref{Propomin}) that if $Z\subset \IR^{d+N}$ is definable in an 
o-minimal structure and each fiber $Z_T=\{\vvx; (T,\vvx)\in Z\}\subset \IR^N$
is bounded then each fiber $Z_T$ lies in $\F_{\ka_Z,M_Z}$ for certain constants $\ka_Z$ and $M_Z$ depending
only on $Z$ but not on $T$. This result provides another rich source of interesting examples, and might be of independent interest.

For $1\leq i\leq n$ let $\pi_i:\E\rightarrow \IR^{\mi}$ be the projection defined by $\pi_i(\vvx)=\vx_i$. 
We fix values $\ka$ and $M$, and we assume
throughout this article that $\Z\subset \IR^N$ is such that for all $1\leq i\leq n$ 
\begin{alignat*}3
(1)\;&\Z\in \F_{\ka,M},\\
(2)\;&\pi_i(\Z)\subset \Byi(\Qi) \text{ for some }\vyi \in \IR^{\mi}.
\end{alignat*}
Here $B_{\vyi}(\Qi)$ denotes the closed Euclidean ball  in $\IR^{\mi}$ about $\vy_i$ of radius $\Qi$.
As is well known (see, e.g., \cite{Spain}) $\partial(\Z)\in \text{ Lip}(N,M,L)$ implies that $\Z$ is measurable.

\subsection{Main results}
Let $(\S,\C)$ be given. For $\Gamma\subset \E$ we introduce the quantities
\begin{alignat*}3
\la_1(\Gamma):=\inf\{|\vvx|; \vvx\in \Gamma\backslash\{\vv0\}\},
\end{alignat*}
and
\begin{alignat*}3
\mu(\Gamma,\ro):=\min\{\la_1(\Gamma\cap \C), \nu(\Gamma,\ro)\}.
\end{alignat*}
If $\mu(\Gamma,\rho)=\infty$ then we interpret $1/\mu(\Gamma,\rho)$ as $0$. 
Finally, we introduce the error term
\begin{alignat*}3
\LEZ:=\left|\#(\Z\cap\La)-\frac{\Vol \Z}{\det \La}\right|.
\end{alignat*}
Our first result is a sharp upper bound for $\LEZ$.
\begin{theorem}\label{Thm1}
Suppose $\La$ is a weakly admissible lattice for $(\S,\C)$, and define $\cthree:=M((1+\ka)N^{2N})^N$. Then we have
\begin{alignat*}3
\LEZ
\leq \cthree \inf_{0<\Ac\leq \Qm}\left(\frac{\Rz}{\mu(\La,\Ac)}+\frac{\Qm}{\Ac}\right)^{N-1}.
\end{alignat*}
\end{theorem}
Considering suitable homogeneously expanding parallelepipeds it is clear that the error term cannot be improved in this generality. 
However, the situation becomes much more interesting when we restrict the sets $\Z$ to aligned boxes. 
In this case Skriganov conjectured \cite[Remark 1.1]{Skriganov1994} that his error term \cite[(1.11) Theorem 1.1]{Skriganov1994} for admissible lattices  (in his sense) is sharp.
Skriganov's conjecture would follow from the expected sharp lower bound for the extremal discrepancy of sequences in the unit cube in $\IR^N$ (see \cite[Remark 2.2]{Skriganov1994}); however, this is a major open
problem in uniform distribution theory, proved only for $N=2$ by Schmidt \cite{SchmidtVII}. Therefore, the sharpness of Skriganov's error term for admissible lattices is known only
for $N=2$.
Here we are able to show that for weakly admissible lattices (in our sense) the error term in Theorem \ref{Thm1} is sharp for $N=2$ and $N=3$.
\begin{theorem}\label{sharp}
Suppose $2\leq n\leq 3$, $\m_i=\beta_i=1$ ($1\leq i\leq n$) (hence $N=n$) and $\C=\{\vvx;\vx_n=\vNull\}$. Then there exists an absolute constant $c_{abs}>0$, a unimodular, weakly admissible lattice $\La$ for $(\S,\C)$,
and a sequence of increasingly  distorted (i.e., $\Rz/\Qm$ tends to zero), aligned boxes $\Z=[-Q_1,Q_1]\times\cdots\times[-Q_n,Q_n]$, satisfying (\ref{Qorder}), whose volume $(2\Rz)^N$ tends to infinity, 
such that  for each box $\Z$
$$\LEZ \geq c_{abs} \inf_{0<\Ac\leq \Qm}\left(\frac{\Rz}{\mu(\La,\Ac)}+\frac{\Qm}{\Ac}\right)^{N-1}.$$
\end{theorem}
Thanks to the good  dependence on the lattice of the  error term in Theorem \ref{Thm1} we are also able to prove asymptotics for the number of primitive lattice points. 

Let $\La$ be a lattice in $\IR^N$.  We say $\vvx\in \La$ is primitive if $\vvx$ is not of the form $k\vvy$ for some $\vvy\in \La$ and some integer $k>1$.
We write
\begin{alignat*}3
\La^*:=\{\vvx\in \La; \vvx\text{ is primitive}\}.
\end{alignat*}
To state our next result let $\T:[0,\infty)\rightarrow [1,\infty)$ be monotonic increasing, and an upper bound for the divisor function, i.e.,
$$\T(k)\geq \sum_{d\mid k}1$$ for all $k\in \IN$.  
Finally, $\zeta(\cdot)$ denotes the Riemann zeta function.
\begin{theorem}\label{Thm2}
Suppose $\La$ is a weakly admissible lattice for $(\S,\C)$. Then there exists a constant $\cone=\cone(N,\ka,M)$, depending only on $N,\ka,M$, such that
\begin{alignat*}3
\left|\#(\Z\cap\La^*)-\frac{\Vol \Z}{\zeta(N)\det \La}\right|
\leq \cone \left(\left(\frac{\Rz}{\mu}+1\right)^{N-1}+\left(\frac{\Rz}{\mu}+1\right)\T\left(\Ad \right)\right),
\end{alignat*}
where 
\begin{alignat*}1
\Ad
=N^{2N+2}(\Rz+|\phi(\vvy)|)\left(\frac{1}{\mu}+\frac{1}{\Rz}\right),
\end{alignat*}
$\mu=\mul$,
and
$|\phi(\vvy)|$ is the Euclidean norm of $(\Rz\vy_1/\Qo,\ldots,\Rz\vy_n/\Qn)\in \E$.
\end{theorem}
Note that for every $a>2$ there is a $b=b(a)\geq \exp(\exp(1))$ such that for $x\geq b$ we can take $T(x)=a^{\frac{\log x}{\log\log x}}$. 
We use $\Rz+|\phi(\vvy)|\leq \Rz(1+|\vvy|/\Qmin)$ and ${1}/{\mu}+{1}/{\Rz}\leq 2/\mu$
to obtain the following corollary.

\begin{corollary}\label{Cor}
Suppose $\La$ is a weakly admissible lattice for $(\S,\C)$ and $a>2$. Then there exists a constant $\ctwo=\ctwo(a,N,\ka,M,|\vvy|)$, depending 
only on $a,N,\ka,M$ and $|\vvy|$ such that for all $\Rz\geq b\mu$ we have
\begin{alignat*}3
\left|\#(\Z\cap\La^*)-\frac{\Vol \Z}{\zeta(N)\det \La}\right|
\leq \ctwo\left(\left(\frac{\Rz}{\mu}\right)^{N-1}+a^{\frac{\log(\eta\Rz/\mu)}{\log\log(\eta\Rz/\mu)}}\left(\frac{\Rz}{\mu}\right)\right),
\end{alignat*}
where 
$\mu=\mul$, and $\eta=1+|\vvy|/\Qmin$.
\end{corollary}

Next we consider applications to Diophantine approximation.
Let $\Theta\in \Mat_{r\times s}(\IR)$ be a matrix with $r$ rows and $s$ columns and suppose that
$\App:[1,\infty)\rightarrow (0,1]$ is a non-increasing function such that  
\begin{alignat}1\label{Diophcond}
|\vp+\Theta\vq|{|\vq|}^{\beta}\geq \App(|\vq|)
\end{alignat}
for every $(\vp,\vq)$ with $\vq\neq \vNull$. 
Let $\vy$ be in $\IR^r$, $Q\geq 1$, and let $0<\epsilon\leq 1$. We consider the system 
\begin{alignat}1
\label{hallo1} 
\vp+\Theta\vq-\vy\in [0,\epsilon]^r\\
\label{hallo2} 
\vq\in [0,Q]^s.
\end{alignat}
Let $N^*_{\Theta,\vy}(\epsilon,Q)$ be the number of $(\vp,\vq)\in \IZ^{r+s}$ that satisfy the above system and 
have coprime coordinates, i.e., $\gcd(p_1,\ldots,p_r,q_1,\ldots,q_s)=1$. 
In the one-dimensional case $r=s=1$ Chalk and Erd\H{o}s \cite{ChalkErdos1959} proved in 1959
 that if $\Theta$ is an irrational number and $\epsilon=\epsilon(\vq)=(1/\vq)(\log \vq/\log\log \vq)^2$ then 
 (\ref{hallo1}) has infinitely many coprime solutions, i.e., 
 $N^*_{\Theta,\vy}(\epsilon,Q)$ is unbounded as $Q$ tends to infinity. 
 No improvements or generalisations 
 have been obtained since. 

The following corollary follows straightforwardly from Corollary \ref{Cor}, and we leave the
proof to the reader. We suppose $\epsilon=\epsilon(Q)$ is a function of $Q$, and that $\epsilon\cdot Q^\beta$ tends to
infinity as $Q$ tends to infinity.
\begin{corollary}\label{Cor2}
Suppose $a>2$. Then, as $Q$ tends to infinity, we have
\begin{alignat*}3
N^*_{\Theta,\vy}(\epsilon,Q)=\frac{\epsilon^r Q^s}{\zeta(r+s)}+O(u^{r+s-1}+ua^{\frac{\log\delta}{\log\log\delta}}),
\end{alignat*}
where 
$u=\left(\frac{\epsilon Q^\beta}{\App(Q)}\right)^{1/(1+\beta)}$, and $\delta=\left(\frac{1}{\App(Q)}\left(\frac{Q}{\epsilon}\right)^\beta\right)^{1/(1+\beta)}$.
\end{corollary}
Corollary \ref{Cor2} also implies new results on how quickly $\epsilon$ can decay so that (\ref{hallo1}) still has infinitely many coprime solutions.
As an example let us suppose that $\Theta$ is a badly approximable matrix so that in (\ref{Diophcond}) we can choose $\beta=\s/\r$ and $\App(\cdot)$ to be constant.
A straightforward computation shows that if $c>2^{(\r\s+\s^2)/(\r^2(\r+\s-1))}$ and $\epsilon=\epsilon(Q)=Q^{-\s/\r}c^{\log Q/\log\log Q}$ then $N^*_{\Theta,\vy}(\epsilon,Q)$ tends to infinity as $Q$ does. In particular, if 
$\epsilon=\epsilon(|\vq|_\infty)=|\vq|_\infty^{-\s/\r}c^{\log |\vq|_\infty/\log\log |\vq|_\infty}$ then (\ref{hallo1}) has infinitely many coprime solutions\footnote{Here $|\cdot|_\infty$ denotes the maximum norm.}.
To the best of the author's knowledge this is the first
such result 
result in arbitrary dimensions.

A similar simple calculation shows that Corollary \ref{Cor2} in conjunction with the classical Khintchine Groshev Theorem 
implies that the same holds true not only for badly approximable matrices $\Theta$ but for almost\footnote{With respect to the Lebesgue measure.} every $\Theta\in \Mat_{r\times s}(\IR)$.

Finally, we mention a connection to a question of Dani, Laurent and Nogueira \cite{DaniLaurentNogueira2014,LaurentNogueira2012}.
Suppose $\epsilon:[1,\infty)\rightarrow (0,1]$ and $Q^{\s-1}\epsilon(Q)^{\r}$ is non-increasing. Dani, Laurent and Nogueira conjecture\footnote{ In fact their conjecture is more general
but the mentioned special case is probably the most natural case.}
\cite[2. paragraph after Theorem 1.1]{DaniLaurentNogueira2014}
that if $\sum_{j\in \IN}j^{\s-1}\epsilon(j)^\r=\infty$
then for almost
every $\Theta\in \Mat_{r\times s}(\IR)$ there exist infinitely many coprime solutions
of (\ref{hallo1}), where again we interpret $\epsilon=\epsilon(|\vq|_\infty)$ as a function evaluated at $|\vq|_\infty$.
We cannot prove this conjecture but, as mentioned before, our result shows at least that
we have infinitely many such solutions for almost every $\Theta$ if  $\epsilon(Q)\gg Q^{-\s/\r}c^{\log Q/\log\log Q}$ and $c>2^{(\r\s+\s^2)/(\r^2(\r+\s-1))}$.

\section{Basic counting principle}\label{countingprinciples}
Let $\Da\geq 2$ be an integer. Let $\Lambda$ be a lattice of rank $\Da$ in $\IR^\Da$. 
Recall that $B_P(R)$ denotes the closed Euclidean ball about $P$ of radius $R$.
We define the successive minima 
 $\lambda_1(\Lambda),\ldots,\lambda_\Da(\Lambda)$ of $\Lambda$ as the successive minima in the sense of Minkowski with respect
to the Euclidean unit ball. That is
\begin{alignat*}3
\lambda_i=\inf \{\lambda; B_0(\lambda)\cap \Lambda
\text{ contains $i$ linearly independent vectors}\}.
\end{alignat*}
\begin{definition}\label{Lip}
Let $M$ be a positive integer, and let $L$ be a non-negative real number.
We say that a set $S$ is in Lip$(\Da,M,L)$ if 
$S$ is a subset of $\IR^\Da$, and 
if there are $M$ maps 
$\phi_1,\ldots,\phi_M:[0,1]^{\Da-1}\longrightarrow \IR^\Da$
satisfying a Lipschitz condition
\begin{alignat*}3
|\phi_i(\vx)-\phi_i(\vy)|\leq L|\vx-\vy| \text{ for } \vx,\vy \in [0,1]^{\Da-1}, i=1,\ldots,M 
\end{alignat*}
such that $S$ is covered by the images
of the maps $\phi_i$. 
\end{definition}
For any set $S$ we write
\begin{equation*}
1^*(S)= 
\begin{cases}  1& \text{if $S\neq \emptyset$,}
\\
0&\text{if $S=\emptyset$.}
\end{cases}
\end{equation*}
We will apply the following basic counting principle.

\begin{lemma}\label{MV_CL}
Let $\Lambda$ be a lattice in $\IR^\Da$
with successive minima $\lambda_1,\ldots, \lambda_\Da$.
Let $S$ be a set in $\IR^\Da$ such that
the boundary $\partial S$ of $S$ is in Lip$(\Da,M,L)$, and suppose $S\subset B_P(L)$ for some point $P$.
Then $S$ is measurable, and moreover,
\begin{alignat*}3
\left|\#(S\cap\Lambda)-\frac{\Vol S}{\det \Lambda}\right|
\leq \ccountlatticepts(\Da) M\left(\left(\frac{L}{\lambda_1}\right)^{\Da-1}+1^*(S\cap \Lambda)\right),
\end{alignat*}
where $\ccountlatticepts(\Da)=\Da^{3\Da^2/2}$.
\end{lemma}
\begin{proof}
By \cite[Theorem 5.4]{art1} the set $S$ is measurable, and moreover,
\begin{alignat}3\label{thm5.4}
\left|\#(S\cap\Lambda)-\frac{\Vol S}{\det \Lambda}\right|
\leq \Da^{3\Da^2/2}M\max_{1\leq j<\Da}\left\{1,\frac{L^j}{\lambda_1\cdots \lambda_j}\right\}.
\end{alignat}
First suppose $L\geq \lambda_1$. Then the lemma follows immediately from (\ref{thm5.4}).
Next we assume $L<\lambda_1$. We distinguish two subcases.
First suppose  $S\cap\Lambda\neq \emptyset$. Then
\begin{alignat*}3
\max_{1\leq j<\Da}\left\{1,\frac{L^j}{\lambda_1\cdots \lambda_j}\right\}=1=1^*(S\cap \Lambda)\leq\left(\frac{L}{\lambda_1}\right)^{\Da-1}+1^*(S\cap \Lambda).
\end{alignat*}
Now suppose $S\cap\Lambda=\emptyset$. 
As $L<\lambda_1$ we get, using Minkowski's second Theorem,  
\begin{alignat*}1 
\left|\#(S\cap\Lambda)-\frac{\Vol S}{\det \Lambda}\right|=\frac{\Vol S}{\det \Lambda}\leq \frac{(2L)^\Da}{\lambda_1\cdots \lambda_\Da}\leq
2^{\Da}\left(\frac{L}{\lambda_1}\right)^{\Da-1}.
\end{alignat*}
This proves the lemma.
\end{proof}

\section{Proof of Theorem \ref{Thm1}}
Let $\thi=\Rz/\Qi$  ($1\leq i\leq n$), and let $\phi$ be the automorphism of $\E$ defined by 
\begin{alignat*}1
\phi(\vvx):=(\tho\vx_1,\ldots,\thn\vx_n).
\end{alignat*}
Set
\begin{alignat*}1
\thm:=\min_{1\leq i\leq n}\thi=\Rz/\Qm.
\end{alignat*}
Note that by (\ref{Qorder}) we have
\begin{alignat}3\label{thetaorder}
\thi\geq 1\;(\text{for  all } i\notin I).
\end{alignat}
Moreover,
\begin{alignat*}3
\prod_{i=1}^{n}\thi^{\beta_i}=1,
\end{alignat*}
and hence,
\begin{alignat}3\label{Nminv}
\Nm(\phi\vvx)=\Nm(\vvx).
\end{alignat}

\begin{lemma}\label{Lip-par}
We have $\partial \phi(\Z)\in$ Lip$(N,M,L)$ for $L=2n^{1/2}\ka\Rz$.
\end{lemma}
\begin{proof}
We have 
\begin{alignat*}1
\phi(\Z)\subset \phi(B_{\vy_1}(\Qo)\times\cdots \times B_{\vy_n}(\Qn))=B_{\tho\vy_1}(\Rz)\times\cdots \times B_{\thn\vy_n}(\Rz),
\end{alignat*}
and hence, $\phi(\Z)\subset B_{\phi\vvy}(\dS\Rz)$. As $\Z\in \F_{\ka,M}$ the claim follows.
\end{proof}

\begin{lemma}\label{Prop1}
The set $\Z$ is measurable and 
\begin{alignat*}3
\left|\#(\Z\cap\La)-\frac{\Vol \Z}{\det \La}\right|
\leq \cfour\left(\left(\frac{\Rz}{\la_1(\phi\La)}\right)^{N-1}+1^*(\phi\Z\cap\phi\La)\right),
\end{alignat*}
where $\cfour=(1+2n^{1/2}\ka)^{N-1}M\ccountlatticepts(N)$.
\end{lemma}
\begin{proof}
Since $\#(\Z\cap\La)=\#(\phi\Z\cap\phi\La)$
and ${\Vol \Z}/{\det \La}={\Vol \phi\Z}/{\det \phi\La}$ this follows immediately from Lemma \ref{MV_CL} and Lemma \ref{Lip-par}.
\end{proof}

\begin{lemma}\label{minimumestimate}
Let $\Ac>0$. Then we have
\begin{alignat*}3
\la_1(\phi\La)\geq \min\{\la_1(\La\cap \C_I), \nu(\La,\Ac), \thm\Ac\}.
\end{alignat*}
\end{lemma}
\begin{proof}
By (\ref{thetaorder}) we have $\theta_i\geq 1$ (for all $i\notin I$).
Moreover, if $\vvx\in \La\cap \C_I$ then $\vx_i=\v0$ (for all $i\in I$), and thus
\begin{alignat*}3
|\phi(\vvx)|^2=\sum_{1\leq i\leq n\atop i\notin I}|\theta_i\vx_i|^2\geq \sum_{1\leq i\leq n\atop i\notin I}|\vx_i|^2=|\vvx|^2.
\end{alignat*}
Hence, if $\vvx\in \La\cap \C_I$ and $\vvx\neq 0$ then
$|\phi(\vvx)|\geq\la_1(\La\cap\C_I)$.

Now suppose that $\vvx\in \La\backslash \C_I$. If $\vvz$ is an arbitrary point in $\E$ then,
by the weighted arithmetic geometric mean inequality, we have  
\begin{alignat*}1
|\vvz|^2=\sum_{i=1}^{n}|\vz_i|^2\geq \frac{1}{\max_i \beta_i}\sum_{i=1}^{n}\beta_i|\vz_i|^2
\geq \frac{\t}{\max_i \beta_i}\left(\prod_{i=1}^{n}|\vz_i|^{2\beta_i}\right)^{\frac{1}{\t}}\geq\Nm(\vvz)^{2/\t},
\end{alignat*}
and thus 
\begin{alignat}3\label{normbound}
|\vvz|\geq  \Nm(\vvz)^{1/\t}.
\end{alignat}
Using  (\ref{normbound}) and (\ref{Nminv}) we conclude that
\begin{alignat*}3
|\phi(\vvx)|\geq \Nm(\phi\vvx)^{1/\t}=\Nm(\vvx)^{1/\t}.
\end{alignat*}
First suppose that $|\vvx|< \Ac$. Then we have by the definition of $\nu(\cdot,\cdot)$
\begin{alignat*}3
\Nm(\vvx)^{1/\t}\geq \nu(\La,\Ac),
\end{alignat*}
and hence $|\phi(\vvx)|\geq \nu(\La,\Ac)$. 
Now suppose $|\vvx|\geq\Ac$. Then we have
\begin{alignat*}3
|\phi(\vvx)|=\thm|(\tho\vx_1/\thm,\ldots,\thn\vx_n/\thm)|\geq \thm|(\vx_1,\ldots,\vx_n)|=\thm |\vvx|\geq \thm \Ac.
\end{alignat*}
This proves the lemma.
\end{proof}
We can now easily finish the proof of Theorem \ref{Thm1}. 
Since, $\thm \Qm=\Rz$ we conclude $\la_1(\phi\La)\geq \min\{\mu(\La,\Ac), \Ac\Rz/\Qm\}$. Thus, we have
\begin{alignat}3\label{errortermbound}
\frac{\Rz}{\la_1(\phi\La)}\leq\frac{\Rz}{\mu(\La,\Ac)}+\frac{\Qm}{\Ac}.
\end{alignat}
The latter in conjunction with Lemma \ref{Prop1} and the fact $\cfour+1=(1+2n^{1/2}\ka)^{N-1}MN^{3N^2/2}+1\leq M((1+\ka)N^{2N})^N=\cthree$ proves the theorem.

\section{Preparations for the M\"obius inversion}\label{1errorterm}

Recall that $\T:[0,\infty)\rightarrow [1,\infty)$ is a monotonic increasing function that is an upper bound for the divisor function, i.e.,  $\T(k)\geq \sum_{d|k}1$ for all $k\in \IN$. In this section $\Da$ is a positive integer. For $A\in \Aut_\Da(\IR)$ we write $\|A\|$ for the (Euclidean) operator norm.
\begin{lemma}\label{1*est}
Let $\Lambda$ be a lattice in $\IR^\Da$, and let $A$ be in $\Aut_\Da(\IR)$ with $A\IZ^\Da=\Lambda$.
Then
\begin{alignat*}3
\#\{k\in \IN; B_P(R)\backslash \{\vNull\}\cap k\Lambda\neq \emptyset\} \leq \T((R+|P|)\|A^{-1}\|)(2R\|A^{-1}\|+1).
\end{alignat*}
\end{lemma}
\begin{proof}
First assume $A=I_\Da$ so that $\Lambda=\IZ^\Da$. Suppose $v=(a_1,\ldots,a_\Da)\in \IZ^\Da$ is non-zero, $kv\in B_P(R)$ and $P=(x_1,\ldots,x_\Da)$. 
Then $ka_i$ lies in $[x_i-R,x_i+R]$ for $1\leq i\leq \Da$. As $v\neq \vNull$ there exists an $i$ with $a_i\neq 0$.
We conclude that $k$ is a divisor of some non-zero integer in $[x_i-R,x_i+R]$. There are at most $2R+1$ integers in this interval, each of which of modulus
at most $R+|P|$. Hence the number of possibilities for $k$ is
$\leq \T(R+|P|)(2R+1)$. This proves the lemma for $A=I_\Da$.
Next note that 
$$\#(B_P(R)\backslash \{\vNull\}\cap k\La)=\#(A^{-1}B_P(R)\backslash \{\vNull\}\cap k\IZ^\Da).$$
Hence, the general case follows from the case $A=I_\Da$ upon noticing 
$A^{-1}B_P(R)\subset B_{A^{-1}(P)}(R\|A^{-1}\|)$, 
and $|A^{-1}(P)|\leq\|A^{-1}\||P|$.
\end{proof}

Next we estimate the operator norm $\|A^{-1}\|$ for a suitable choice of $A$.
\begin{lemma}\label{Operatornormest}
Let $\Lambda$ be a lattice in $\IR^\Da$.
There exists $A\in \Aut_\Da(\IR)$ with $A\IZ^\Da=\Lambda$ and
\begin{alignat*}3
\|A^{-1}\|\leq \frac{\copnorm(\Da)}{\lambda_1},
\end{alignat*}
where $\copnorm(\Da)=\Da^{2\Da+1}$.
\end{lemma}
\begin{proof}
Any lattice $\Lambda$ in $\IR^\Da$ has a basis $v_1,\ldots,v_\Da$ with $\frac{|v_1|\cdots|v_\Da|}{|\det[v_1\ldots v_\Da]|}\leq \Da^{2\Da}$, see, e.g., \cite[Lemma 4.4]{art1}.
Let $A$ be the matrix that sends the canonical basis $e_1,\ldots,e_n$ to $v_1,\ldots,v_n$. Now suppose $A^{-1}$ sends
$e_i$ to $(\rho_1,\ldots,\rho_n)$ then by Cramer's rule
\begin{alignat*}3
|\rho_j|=&\left|\frac{\det[v_1\ldots e_i\ldots v_\Da]}
{\det[v_1\ldots v_j\ldots v_\Da]}\right|\leq\frac{|\det[v_1\ldots e_i\ldots v_\Da]|}
{|v_1|\cdots|v_j|\cdots|v_\Da|}\Da^{2\Da}.
\end{alignat*}
Now we apply Hadamard's inequality to obtain 
\begin{alignat*}3
\frac{|\det[v_1\ldots e_i\ldots v_\Da]|}
{|v_1|\cdots|v_j|\cdots|v_\Da|}
\leq \frac{|v_1|\cdots|e_j|\cdots|v_\Da|}{|v_1|\cdots|v_i|\cdots|v_\Da|}
=\frac{1}{|v_i|}\leq \frac{1}{\lambda_1}.
\end{alignat*}
Next we use that for a $\Da\times \Da$ matrix $[a_{ij}]$ with real entries we have $\|[a_{ij}]\|\leq \sqrt{\Da}\max_{ij}|a_{ij}|$, and this proves the lemma.
\end{proof}
We combine the previous two lemmas. 
\begin{lemma}\label{mainlemma}
Let $\Lambda$ be a lattice in $\IR^\Da$, and let $\lambda_1=\lambda_1(\Lambda)$.
Then 
\begin{alignat*}3
\sum_{k=1}^{\infty}1^*(B_P(R)\backslash\{\vNull\}\cap k\Lambda )\leq \T\left(\copnorm(\Da)\left(\frac{R+|P|}{\lambda_1}\right)\right)\left(\frac{2\copnorm(\Da) R}{\lambda_1}+1\right).
\end{alignat*} 
\end{lemma}
\begin{proof}
Note that $\sum_{k=1}^{\infty}1^*(B_P(R)\backslash\{\vNull\}\cap k\Lambda)=\#\{k\in \IN;  B_P(R)\backslash\{\vNull\}\cap k\Lambda\neq \emptyset\}$. 
Hence, the lemma follows immediately from Lemma \ref{1*est} and Lemma \ref{Operatornormest}.
\end{proof}

\section{Proof of Theorem \ref{Thm2}}\label{proofthm2}

Set
\begin{alignat*}3
\Z^*:=\Z\backslash\{\vv0\},
\end{alignat*}
and
\begin{alignat*}3
\R:=\dS\Rz.
\end{alignat*}
\begin{lemma}\label{Prop2}
We have
\begin{alignat*}3
\left|\#(\Z^*\cap\La)-\frac{\Vol \Z}{\det \La}\right|
\leq \cfive \left(\left(\frac{\Rz}{\la_1(\phi\La)}\right)^{N-1}+1^*(B_{\phi(\vvy)}(\R)\backslash\{\vv0\}\cap\phi\La)\right),
\end{alignat*}
where $\cfive=(1+2n^{1/2}\ka)^{N-1}(M+1)\ccountlatticepts(N)$.
\end{lemma}
\begin{proof}
Lemma \ref{Lip-par} implies that $\partial \Z^*\in$ Lip$(N,M+1,L)$ with $L=2n^{1/2}\ka\Rz$.
As noted in the proof of the latter lemma we have $\phi(\Z^*)\subset B_{\phi\vvy}(\R)\backslash\{\vv0\}$.
We conclude as in Lemma \ref{Prop1}.
\end{proof}

For $\vvx\in \La\backslash\{\vv0\}$ we define $\gcd(\vvx):=d$ if $\vvx=d\vvx'$ for some $\vvx'\in \La$ but $\vvx\neq k\vvx'$ for all integers $k>d$ and all $\vvx'\in\La$.
(An equivalent definition
is $\gcd(A\vvz):=\gcd(\vvz)$, where $\vvz\in \IZ^N$, $\gcd(\vvz):=\gcd(z_1,\ldots,z_N)$, and $\La=A\IZ^N$.)
Next we define
\begin{alignat*}1
\Fd=\{\vvx\in \La\cap\Z^*;  \gcd(\vvx)=d\}.
\end{alignat*}
In particular, $\La^*\cap\Z=\Fo$.
Then for $k\in \IN$ we have the disjoint union
\begin{alignat*}1
\bigcup_{k\mid d}\Fd=k\La\cap\Z^*.
\end{alignat*}
If $\vvx=k\vvx'$ lies in $k\La\cap\Z^*$ then $k\phi\vvx'$ lies in $k\phi\La\cap B_{\phi(\vvy)}(\R)$, and hence 
$$k\leq \frac{\R+|\phi(\vvy)|}{\la_1(\phi\La)} \leq \frac{\R+|\phi(\vvy)|}{\mul} + \frac{\R+|\phi(\vvy)|}{\Rz}=:\Ab,$$
where for the second inequality we have applied Lemma \ref{minimumestimate}.
We use the M\"obius function $\mu(\cdot)$ and the M\"obius inversion formula to get
\begin{alignat*}1
\#(\La^*\cap\Z)=\#\Fo=\sum_{k=1}^{\infty}\mu(k)\sum_{d\atop k|d}\#\Fd=\sum_{k=1}^{[\Rt]}\mu(k)\sum_{d\atop k|d}\#\Fd=\sum_{k=1}^{[\Rt]}\mu(k)\#(k\La\cap\Z^*).
\end{alignat*}
For the rest of this section we will write $g\ll h$ to mean there exists a constant $c=c(N,M,\ka)$ such that $g\leq ch$. 
Applying Lemma \ref{Prop2} with $\La$
replaced by $k\La$ yields
\begin{alignat*}3
&\left|\#(\Z\cap\La^*)-\frac{\Vol \Z}{\zeta(N)\det \La}\right|
\ll\\
&\sum_{k=1}^{[\Rt]}\left(\frac{\Rz}{k\la_1(\phi\La)}\right)^{N-1}+\sum_{k=1}^{[\Rt]} 1^*(B_{\phi(\vvy)}(\R)\backslash\{\vv0\}\cap k\phi\La)+\sum_{k>\Rt}\frac{\Vol \Z}{k^N\det \La}.
\end{alignat*}
First we note that 
\begin{alignat*}3
\sum_{k>\Rt}k^{-N}\leq \sum_{k\geq\max\{\Rt,1\}}k^{-N}\ll \max\{\Rt,1\}^{1-N}\leq \max\{\frac{\R}{\la_1(\phi\La)},1\}^{1-N},
\end{alignat*}
and moreover, 
\begin{alignat*}3
\frac{\Vol \Z}{\det \La}=
\frac{\Vol \phi\Z}{\det \phi\La}\leq
\frac{\Vol B_{\vv0}(\R)}{\det \phi\La}\ll
\frac{\R^N}{\la_1(\phi\La)^N}. 
\end{alignat*}
Combining both with (\ref{errortermbound}) yields
\begin{alignat*}3
\sum_{k>\Rt}\frac{\Vol \Z}{k^N\det \La}\ll \frac{\R}{\la_1(\phi\La)}\ll \frac{\Rz}{\la_1(\phi\La)} 
\leq \frac{\Rz}{\mul}+1.
\end{alignat*}
Next we note that by Lemma \ref{mainlemma}
\begin{alignat*}3
\sum_{k=1}^{[\Rt]}1^*(B_{\phi(\vvy)}(\R)\backslash\{\vv0\}\cap k\phi\La)\leq 
\T\left(\copnorm(N)\frac{\R+|\phi(\vvy)|}{\la_1(\phi(\La))} \right)\left(\frac{2\copnorm(N)\R}{\la_1(\phi(\La))}+1\right).
\end{alignat*}
Moreover,
\begin{alignat*}3
\left(\frac{2\copnorm(N)\R}{\la_1(\phi(\La))}+1\right)\ll \frac{\Rz}{\mul}+1,
\end{alignat*}
and 
\begin{alignat*}3
\frac{\R+|\phi(\vvy)|}{\la_1(\phi(\La))} \leq \frac{\R+|\phi(\vvy)|}{\mul} + \frac{\R+|\phi(\vvy)|}{\Rz}=\Ab.
\end{alignat*}
Since $\copnorm(N)\Ab<\Ad$ we conclude that
\begin{alignat*}3
\sum_{k=1}^{[\Rt]}1^*(B_{\phi(\vvy)}(\R)\backslash\{\vv0\}\cap k\phi\La)\ll
\T\left(\Ad\right)\left(\frac{\Rz}{\mul}+1\right).
\end{alignat*}

Finally, \begin{alignat*}3
\sum_{k=1}^{[\Rt]}\left(\frac{\Rz}{k\la_1(\phi\La)}\right)^{N-1}\ll \left(\frac{\Rz}{\mul}+1\right)^{N-1}\sum_{k=1}^{[\Rt]}k^{1-N}\ll \left(\frac{\Rz}{\mul}+1\right)^{N-1}\L^*,
\end{alignat*}
where
\begin{equation*}
\L^*= 
\begin{cases}  \max\{\log(\Ab),1\} & \text{if $N=2$,}
\\
1&\text{if $N>2$.}
\end{cases}
\end{equation*}
If $N>2$ then $\L^*=1$ and we are done. So  suppose $N=2$. Hence $\copnorm(N)=32$.
By assumption $\T(x)\geq 1$ so that $\L^*\leq T(\copnorm(N)\Ab)$ for $\Ab\leq\exp(1)$. Now suppose $\Ab>\exp(1)$.
Since $T$ is monotonic and $2^{[\log_2[32 \Ab]]}\leq 32\Ab$ we have $\T(32\Ab)\geq [\log_2[32\Ab]]+1\geq \log_2(32\Ab-1)\geq \log\Ab$. 
Thus, $\L^*\leq T(\copnorm(N)\Ab)\leq T(\Ad)$. This finishes the proof.

\section{Lower bounds for the error term}\label{lowerboundserrorterm}
The main goal of this section is to prove Theorem \ref{sharp}. 
Throughout this section we assume that 
$\m_i=\beta_i=1$ ($1\leq i\leq n$), so that $N=n=\t\geq 2$, and that  $\La$ is a unimodular weakly admissible for $(\S,\C)$
but not admissible for $(\S,\C)$.
To simplify the notation we write $\Nmm(\cdot):=\Nm(\cdot)$ and $\nu(\cdot):=\nu(\La,\cdot)$.

Let $\k\geq 1$ be a constant, and $\{\vvx_j\}_{j=1}^\infty=\{(x_{j1},\ldots,x_{jn})\}_{j=1}^\infty$ be a sequence of pairwise distinct elements in $\La\backslash\C$ 
satisfying
\begin{alignat*}1
\Nmm(\vvx_j)\leq \k\nu(|\vvx_j|)^n.
\end{alignat*}
We define 
\begin{alignat*}1
N_j&:=a\nu(|\vvx_j|)^{-n},\\
\Zj&:=N_j\Cx,\\
c_j&:={\la_{n-1}(\La,\Cx)},
\end{alignat*}
where $a>0$ is a constant which will be specified later, $\Cx$ denotes the $\vv0$-centered box 
$$\Cx:=[-|x_{j1}|,|x_{j1}|]\times\cdots\times[-|x_{jn}|,|x_{jn}|],$$ 
and $\la_{i}(\La,\Cx)$ are
the corresponding successive minima.
For $1\leq i\leq n$ we choose the minimal eligible values $Q_i=N_j|x_{ji}|$ for the set $\Zj$, so that\footnote{To simplify the notation we suppress the dependence on $j$ and we simply write $Q_i$ and $\Rz$.} 
\begin{alignat}1\label{Rzest}
\Rz\leq (a\k)^{\frac{1}{n}} N_j^{\frac{n-1}{n}}.
\end{alignat}
We also assume that our sets $\Zj$ satisfy the condition  (\ref{Qorder}), i.e.,
\begin{alignat*}1
\Qi\leq \Rz \;(\text{for all } i\notin I).
\end{alignat*}
\begin{lemma}\label{lemmasixone}
We have 
\begin{alignat*}1
\#(\Zj\cap\La)-\Vol\Zj\geq (N_j/(c_jn))^{n-1}-2^na\k N_j^{n-1}.
\end{alignat*}
Moreover, $N_j$ tends to infinity and $\Rz/\Qm$ tends to zero.
\end{lemma}
\begin{proof}
Let $v_1,\ldots,v_{n-1}$ be linearly independent lattice points in $\la_{n-1}(\La,\Cx)\Cx$. Then the lattice points
$\sum_{l=1}^{n-1}m_lv_l$ with $-N_j/(c_j n)\leq m_l \leq N_j/(c_j n)$ are all distinct and lie all in $\Zj$. Since $2[N_j/(c_jn)]+1\geq N_j/(c_jn)$ the claimed inequality follows at once.
Recall that  $\La$ is not admissible, and hence $N_j$ tends to infinity, and thus $\Rz/\Qm$ tends to zero. 
\end{proof}
We now make the crucial assumption 
that the $n-1$-th successive minimum $c_j$
is uniformly bounded\footnote{Note that $\la_1(\La,\Cx)\leq 1$ by definition of the box $\Cx$. On the other hand $\Vol\Cx$ tends to zero, so that by Minkowski's second Theorem
$\la_n(\La,\Cx)\rightarrow \infty$ as $j$ tends to infinity.} in $j$.
\begin{lemma}
Suppose there exists a constant $c_\La\geq 1$ such that
\begin{alignat}2
\label{succminbound}
c_j\leq c_\La
\end{alignat}
for all $j$, and take $a:=1/(4\k(2c_\La n)^{n-1})$. Then we have 
\begin{alignat}1\label{LEZbound}
\LEZi\geq \#(\Zj\cap\La)-\Vol\Zj\geq (c_\La n)^{-n}N_j^{n-1}.
\end{alignat}
\end{lemma}
\begin{proof}
This follows immediately from Lemma \ref{lemmasixone}.
\end{proof}
Next we prove a general criterion for $\La$ under which we have
\begin{alignat}1\label{errorlowerbound}
\#(\Zj\cap\La)-\Vol\Zj\geq \cet \inf_{0<\Ac\leq \Qm}\left(\frac{\Rz}{\mu(\La,\Ac)}+\frac{\Qm}{\Ac}\right)^{N-1}
\end{alignat}
with a certain constant $\cet>0$. 

\begin{proposition}\label{lowerboundcrit}
Suppose that the condition (\ref{succminbound}) and 
\begin{alignat}2
\label{growcond}
\nu\left(\frac{|\vvx_j|}{\nu(|\vvx_j|)^n}\right)\geq \con \nu(|\vvx_j|)
\end{alignat}
for some constant $\con>0$ hold true.
Then there exists $\cet=\cet(\k,c_\La,n,\con)>0$ such that (\ref{errorlowerbound}) holds true for all $j$ large enough.
\end{proposition}
\begin{proof}
We have $\Qm\leq N_j|\vvx_j|$, and so ignoring the first few members of the sequence $\vvx_j$,  
we can assume that 
$$\mul\geq\nu(N_j|\vvx_j|)=\nu(a|\vvx_j|/\nu(|\vvx_j|)^n)\geq \nu(|\vvx_j|/\nu(|\vvx_j|)^n)\geq \con\nu(|\vvx_j|).$$ 
Hence,
$$\inf_{0<\Ac\leq \Qm}\left(\frac{\Rz}{\mu(\La,\Ac)}+\frac{\Qm}{\Ac}\right)\leq \left(\frac{\Rz}{\mul}+1\right)\leq\left(\frac{\Rz}{\con \nu(|\vvx_j|)}+1\right)\leq (2\k^{1/n}/\con)N_j$$
for all $j$ large  enough.
This, in conjunction with (\ref{LEZbound}), shows that (\ref{errorlowerbound}) holds true.
\end{proof}
For the rest of this section we assume that 
\begin{alignat}3\label{Cchoice}
\C=\{\vvx;\vx_n=\v0\}.
\end{alignat}
We now apply Proposition \ref{lowerboundcrit} to prove the case $n=2$ in Theorem \ref{sharp}. 
\begin{proposition}\label{Thmsharptwo}
Suppose $n=2$. Then there exists a unimodular, weakly admissible lattice $\La$ for $(\S,\C)$,
and a sequence of increasingly  distorted (i.e., $\Rz/\Qm$ tends to zero), aligned boxes $\Z=[-Q_1,Q_1]\times[-Q_2,Q_2]$ whose volume $(2\Rz)^2$ tends to infinity
such that  
$$\LEZ \geq c_{abs} \inf_{0<\Ac\leq \Qm}\left(\frac{\Rz}{\mu(\La,\Ac)}+\frac{\Qm}{\Ac}\right),$$
where $c_{abs}>0$ is an absolute constant.
\end{proposition}
\begin{proof}
Let $\alpha$ be an irrational real number, and consider the lattice $\La$ given by the vectors $(p-q\alpha,q)$ with $p,q\in \IZ$.
Then $\La$ is unimodular and weakly admissible for $(\S,\C)$.
To choose an appropriate $\alpha$ we consider its continued fraction expansion $\alpha=[a_0,a_1,a_2,\ldots]$. Using the recurrence relation $q_{j+1}=a_{j+1}q_j+q_{j-1}$
for the denominator $q_j$ of the $j$-th convergent $p_j/q_j$ (in lowest terms) we can define $\alpha$ by setting $a_0=a_1=1$ (so that $q_0=q_1=1$)
and $a_{j+1}=[\log q_j]+1$. Next we note that
$a_{j+1}=[\log(a_jq_{j-1}+q_{j-2})]+1\leq \log((a_j+1)q_{j-1})+1\leq \log(a_j+1)+a_j+1\leq 3a_j$. Similarly we find $a_j+\log a_j-1\leq a_{j+1}$, and hence,
$$a_j+\log a_j-1\leq a_{j+1}\leq 3a_j.$$
Put $\vvx_j=(p_j-q_j\alpha,q_j)\in \La\backslash\C$ so that $|\vvx_j|>|\vvx_{j-1}|$, at least for $j$ large enough.
From the theory of continued fractions we know that for
$\vvx\in \La\backslash\C$ the inequality $\Nmm(\vvx)<1/2$ implies that $\vvx=c\vvx_j$ for some non-zero integer $c$ and $j\in \IN$.
We conclude that for all sufficiently large $\rho$ we have $\nu(\rho)^2=\Nmm(\vvx_j)$ for some $j$.
Also by the theory of continued fractions we know that  
$$1/(a_{j+1}+2)<\Nmm(\vvx_j)<1/a_{j+1}.$$  
Since $a_{j}>a_{j-1}+2$ we conclude $\Nmm(\vvx_{j-1})<\Nmm(\vvx_{j-2})$ and thus 
$$\Nmm(\vvx_{j-1})=\nu(|\vvx_j|)^2$$ 
for $j$ large enough; so we can take $\k=1$. We also easily find that $|\vvx_j|/\nu(|\vvx_j|)^2\leq |\vvx_{j+1}|$ for $j$ large enough.
It is now straightforward to verify (\ref{growcond}).
Moreover, for $j$ large enough, (\ref{Qorder}) holds true, and so $\Zj$ is an eligible set.
Since $n=2$ we automatically have (\ref{succminbound}) with $c_\La=1$.
Hence we can apply Proposition \ref{lowerboundcrit}. Finally, we note that $\Vol \Zj=4N_j^2\Nmm(\vvx_j)=(2a)^2\Nmm(\vvx_{j-1})^{-2}\Nmm(\vvx_j)\geq 2^{-6}a_j^2/(a_{j+1}+2)$ which tends to infinity, and moreover, that
the boxes $\Zj$ are increasingly distorted by Lemma \ref{lemmasixone}. This completes the proof.
\end{proof}
Next we prove the case $n=3$ in Theorem \ref{sharp}. 
This case does not rely on Proposition \ref{lowerboundcrit}.
\begin{proposition}\label{lowerboundexample}
Suppose $n=3$. Then there exists a unimodular, weakly admissible lattice $\La$ for $(\S,\C)$, 
and a sequence of increasingly  distorted, aligned boxes $\Z=[-Q_1,Q_1]\times[-Q_2,Q_2]\times[-Q_3,Q_3]$ whose volume $(2\Rz)^3$ tends to infinity
such that  
$$\LEZ \geq c_{abs} \inf_{0<\Ac\leq \Qm}\left(\frac{\Rz}{\mu(\La,\Ac)}+\frac{\Qm}{\Ac}\right)^2,$$
where $c_{abs}>0$ is an absolute constant.
\end{proposition}
\begin{proof}
Let $\alpha=[a_0,a_1,a_2,\ldots]$ be a badly approximable real number, so that
the partial quotients $a_i$ are bounded.  We set $\a=\max a_i$, and we
consider the lattice
\begin{alignat}3\label{latticeexample}
\La=\{(p_1-q\alpha,p_2-q\alpha,q); p_1,p_2,q\in \IZ\}.
\end{alignat}
The lattice $\La$ is unimodular and weakly admissible for $(\S,\C)$.
In this proof we write $h\ll g$ to mean $h\leq c g$ for a constant $c=c(\a)$
depending only on $\a$. 
First we note that
\begin{alignat*}3
\Nmm(\vvx)\gg |\vvx|^{-1}
\end{alignat*}
for every $\vvx\in \La\backslash \C$. Hence,
\begin{alignat}3\label{nulower}
\nu(\rho)\gg \rho^{-1/3}.
\end{alignat}
Now suppose $p_j/q_j$ is the $j$-th convergent of $\alpha$, and put $\vvx_j=(p_j-q_j\alpha,p_j-q_j\alpha,q_j)\in \La\backslash\C$.
Then, for $j$ large enough, (\ref{Qorder}) holds true, and so $\Zj$ is an eligible set.
Since 
\begin{alignat*}3
\Nmm(\vvx_j)\ll |\vvx_j|^{-1},
\end{alignat*}
we also conclude that there exists $\k=\k(\a)\geq 1$ such that 
\begin{alignat*}3
\Nmm(\vvx_j)\leq\k\nu(|\vvx_j|)^{3}.
\end{alignat*}
Since $q_{j+1}=a_{j+1}q_j+q_{j-1}$ we get $q_{j+1}\ll q_j$ and, as is wellknown,  $|p_{j+1}-q_{j+1}\alpha|<|p_j-q_j\alpha|$.
Furthermore, $(p_j,q_j)$ and $(p_{j+1},q_{j+1})$ are linearly independent, and thus $\vvx_j$ and $\vvx_{j+1}$ are linearly independent. Hence, we conclude
$$c_j:=\la_2(\La,\Cx)\ll 1,$$
and thus, by virtue of (\ref{LEZbound}), we get $\LEZi\gg N_j^2$.
Moreover, for $j$ sufficiently large, we have 
\begin{alignat}3\label{normcomp}
|\vvx_{j-1}|<|\vvx_{j}|\ll |\vvx_{j-1}|,
\end{alignat}
and thus
\begin{alignat}3\label{nuupper}
\nu(|\vvx_j|)\leq \Nmm(\vvx_{j-1})^{1/3}\ll |\vvx_{j-1}|^{-1/3}\ll |\vvx_{j}|^{-1/3}.
\end{alignat}
Combining  (\ref{nulower}), (\ref{normcomp}) and (\ref{nuupper}) implies that 
\begin{alignat*}3
\rho^{-1/3}\ll\nu(\rho)\ll \rho^{-1/3}.
\end{alignat*}
Therefore, we have
\begin{alignat*}3
N_j\ll\nu(|\vvx_j|)^{-3}\ll |\vvx_j|\ll q_j\leq |\vvx_j| \ll\nu(|\vvx_j|)^{-3}\ll N_j.
\end{alignat*}
Thus, $N_j^2\ll \Qm=N_jq_j\ll N_j^2$, and due to (\ref{Rzest}), $\Rz\ll N_j^{2/3}$.
Hence, with $B=N_j$ we have  
\begin{alignat*}1
\frac{\Rz}{\nu(\Ac)}\ll \frac{\Qm}{\Ac},
\end{alignat*}
and thus for all $j$ large enough
\begin{alignat*}3
\inf_{0<\Ac\leq \Qm}\left(\frac{\Rz}{\mu(\La,\Ac)}+\frac{\Qm}{\Ac}\right)^{2}\ll\left(\frac{\Qm}{\Ac}\right)^{2}\ll N_j^{2}\ll \LEZi. 
\end{alignat*}
Hence, we have shown that (\ref{errorlowerbound}) holds true. Finally, we observe that $\Vol \Zj=8N_j^3\Nmm(\vvx_j)\gg N_j^2$ which, due to Lemma \ref{lemmasixone}, completes the proof.
\end{proof}

\section{$\mathcal{F}_{\ka,M}$ - Families via o-minimality}\label{omin}
In this section let $d\geq 1$ and $\Da\geq 2$ both be integers. For $Z\subset \IR^{d+\Da}$ 
and $T\in \IR^d$ we write $Z_T=\{x\in \IR^\Da; (T,x)\in Z\}$ and call this the fiber of $Z$ above $T$. 
For the convenience of the reader we quickly recall the definition of an o-minimal structure
following \cite{PilaWilkie}.
For more details we refer to \cite{Wilkie2007, PilaWilkie} and \cite{vandenDries1998}.
\begin{definition}\label{defomin}
A structure (over $\IR$) is a sequence $\mathcal{S}=(\mathcal{S}_n)_{n\in \IN}$ of families of subsets in $\IR^n$ such that for each $n$:
\begin{enumerate}
\item $\mathcal{S}_n$ is a boolean algebra of subsets of $\IR^n$ (under the usual set-theoretic operations).
\item $\mathcal{S}_n$ contains every semi-algebraic subset of $\IR^n$.
\item If $A \in \mathcal{S}_n$ and  $B\in \mathcal{S}_{m}$ then $A\times B \in \mathcal{S}_{n+m}$.
\item If $\pi: \IR^{n+m}\rightarrow \IR^n$ is the projection map onto the first $n$ coordinates and $A \in \mathcal{S}_{n+m}$ then $\pi(A) \in \mathcal{S}_n$.
\end{enumerate}
An o-minimal structure (over $\IR$) is a structure (over $\IR$) that additionally satisfies: 
\begin{enumerate}
\setcounter{enumi}{4}
\item The boundary of every set in $\mathcal{S}_1$ is finite.
\end{enumerate}
\end{definition} 
The archetypical example of an o-minimal structure is the family of all semi-algebraic sets.

Following the usual convention, we say a set $A$ is definable (in $\mathcal{S}$) if it lies in some $\mathcal{S}_n$.
A map $f:A\rightarrow B$ is called definable if its graph $\Gamma(f):=\{(x,f(x)); x\in A\}$ is a definable set.
\begin{proposition}\label{Propomin}
Suppose $Z\subset \IR^{d+\Da}$  is definable in an o-minimal structure over $\IR$,
and assume further that all fibers $Z_T$ are bounded sets.
Then there exist constants $\ka_Z$ and $M_Z$ depending only on $Z$ (but independent of $T$) such that the fibers
$Z_T$ lie in $\mathcal{F}_{\ka_Z,M_Z}$ for all $T\in \IR^d$.
\end{proposition}
Suppose the set $Z$ is defined by the inequalities
\begin{alignat}1\label{Zfunc}
f_1(T_1,\ldots,T_d,x_1,\ldots,x_\Da)\leq 0,\ldots, f_k(T_1,\ldots,T_d,x_1,\ldots,x_\Da)\leq 0,
\end{alignat}
where the $f_i$ are certain real valued functions on $\IR^{\Da+d}$. If all these functions $f_i$ are definable in 
a common o-minmal structure then we can apply Proposition \ref{Propomin}. This happens for instance
if the $f_i$ are restricted analytic functions\footnote{By a restricted analytic function we mean a real valued function on $\IR^n$, which is zero outside of $[-1,1]^n$, and is the restriction to $[-1,1]^n$ 
of a function, which is real analytic on an open neighborhood of $[-1, 1]^n$.} or polynomials in $z_1,\ldots,z_{d+\Da}$ and each $z_i\in \{T_m,\exp(T_m),x_l,\exp(x_l); 1\leq m\leq d, 1\leq l\leq \Da\}$. For more details and examples
we refer to \cite{Wilkie2007, Scanlon2011, Scanlon2016}.

For the proof of Proposition \ref{Propomin} we shall need the following lemma. 
We are grateful to Fabrizio Barroero for alerting us  to Pila and Wilkies Reparametrization Lemma for definable families
and its relevance for the lemma.
\begin{lemma}\label{ominlemma}
Suppose $Z\subset \IR^{d+\Da}$  is definable in an o-minimal structure over $\IR$,
and assume further that all fibers $Z_T$ are bounded sets. 
Then there exist constants $\ka_Z$ and $M_Z$ depending only on $Z$  such that the boundary
$\partial Z_T$ lies in Lip$(\Da,M_Z,\ka_Z\cdot \diam(Z_T))$ for every $T\in \IR^d$.
\end{lemma}
\begin{proof}
First note that if $\#Z_T\leq 1$ then $\partial Z_T$ lies in Lip$(\Da,1,0)$. Hence, it suffices to prove the claim for those $T$
with $\#Z_T\geq 2$. By replacing $Z$ with the definable set $\{(T,x)\in Z; (\exists x,y\in Z_T)(x\neq y)\}$ we can assume that 
$\#Z_T\geq 2$ for all $T\in \pi(Z)$, where $\pi$ is the projection onto the first $d$ coordinates.
We use the existence of definable Skolem functions. By \cite[Ch.6, (1.2) Proposition]{vandenDries1998} there exists a definable map
$f:\pi(Z)\rightarrow \IR^d$ whose graph $\Gamma(f)\subset Z$.
The proof of said  (1.2) Proposition actually shows that there is an algorithmic way to construct the Skolem function f. We will use the fact that this choice of f is determined by 
$Z$ and $\pi$ and hence can be seen as part of the data of $Z$. 

Now we consider the set $Z'=\{(T,y); (T,x)\in Z, y=x-f(T)\}$. This set is again definable and each non-empty fiber contains the origin, i.e.,
$0\in Z'_T$ for all $T\in \pi(Z)$. Next we scale the fibers and translate by the point $y_0=(-1/2)(1,\ldots,1)\in \IR^\Da$ to get a new definable 
set whose fibers all lie in $(0,1)^\Da$. We put $Z''=\{(T,z); (T,y)\in Z', z=(3\cdot\diam(Z'_T))^{-1}y-y_0\}$
(recall that $\diam(Z'_T)=\diam(Z_T)>0$ since $Z_T$ has at least two points).
We note that the graph of the function $T\rightarrow \diam(Z_T)$ from $\pi(Z)$ to $\IR$ is given by
$$\{(T,t)\in \pi(Z)\times\IR; \phi(T,t)\land \lnot((\exists u\in \IR)(\phi(T,u) \land u<t)\},$$
where $\phi(T,t)$ stands for $(\forall x,y \in Z_T)(|x-y| \leq t)$.
This shows that the aforementioned map is definable and hence, so is $Z''$. Also we have $Z''_T\subset (0,1)^\Da$ for all $T$.
By \cite[Lemma 3.15]{BarroeroWidmer} the set $Z'''=\{(T,w); w\in \partial Z''_T\}$  is also definable. The fibers of a definable set are again definable (cf. \cite[Lemma 3.1]{BarroeroWidmer}),
and hence by \cite[Ch.4, (1.10) Corollary]{vandenDries1998} we have $\dim(\partial Z''_T)\leq\Da-1$.
From Pila and Wilkie's Reparameterization Lemma for definable families \cite[5.2. Corollary]{PilaWilkie}
we conclude\footnote{Using that the partial derivatives are uniformly bounded we can extend the domain of the parametrisation to $[0,1]^{\Da-1}$ without altering the Lipschitz constant.} that $\partial Z''_T$ lies in Lip$(\Da,M_{Z'''},\ka_{Z'''})$ for all $T\in \IR^d$ with certain constants $\ka_{Z'''}$ and $M_{Z'''}$.
Rescaling and retranslating gives $\partial Z_T\in$ Lip$(\Da,M_{Z'''},\ka_{Z'''}\cdot \diam(Z_T))$. Finally, we note that $Z'''$ depends only 
on $Z$ and $f$ which itself can be seen as part of the data of $Z$, so that the constants $\ka_{Z'''}$ and $M_{Z'''}$ may be chosen 
to depend only on $Z$. This completes the proof of the lemma.
\end{proof}

We can now prove Proposition \ref{Propomin}. Consider the set 
$$Z'''':=\{(\varphi,T,x); \varphi \in \Aut_\Da(\IR), x\in \varphi(Z_T)\}.$$
This set is definable in the given o-minimal structure, and we have $Z''''_{(\varphi,T)}=\varphi(Z_T)$.
Applying Lemma \ref{ominlemma} to the fibers $Z''''_{(\varphi,T)}$ we conclude that there exist constants $\ka_{Z''''}$ and $M_{Z''''}$ such that
 $\partial \varphi(Z_T)$ lies in Lip$(\Da,M_{Z''''},\ka_{Z''''}\cdot \diam(\varphi (Z_T)))$ for all $(\varphi,T)\in \Aut_\Da(\IR)\times\IR^d$.
Note that $Z''''$ depends only on $Z$ so that $M_{Z''''},\ka_{Z''''}$ are depending only on $Z$, and this completes
the proof of Proposition \ref{Propomin}.

\section*{Acknowledgements}
It is my pleasure to thank Fabrizio Barroero, Michel Laurent, Arnaldo Nogueira, Damien Roy, and Maxim Skriganov for helpful discussions.
I completed this article during a visiting professorship at Graz University of Technology, and I  thank the Institute of Analysis and Number Theory for its hospitality.

\bibliographystyle{amsplain}
\bibliography{literature}

\end{document}